\documentclass[12pt, a4paper]{article}
\usepackage{fullpage}

\usepackage{latexsym, amssymb, amsmath, amsthm, amsfonts}
\usepackage{appendix}
\usepackage{comment}
\specialcomment{draft}{\begingroup\ttfamily\footnotesize Commentary: }{\endgroup}
\usepackage{bbm}

\usepackage{graphicx}
\usepackage{caption}
\usepackage{subcaption}

\usepackage{placeins}

\theoremstyle{plain}
\newtheorem{thm}{Theorem}[section]
\newtheorem{lem}[thm]{Lemma}
\newtheorem{corollary}[thm]{Corollary}

\newcommand{\re}{\mathrm{e}} 
\newcommand{\rP}{\mathrm{P}} 
\newcommand{\rE}{\mathrm{E}} 
\newcommand{\keywords}[1]
           {\begin{center}
            \begin{minipage}{315.83pt}
            \small
            \noindent \emph{Keywords:}~{\textrm{#1}}
            \end{minipage}
            \end{center}
            \normalsize
           }
\newcommand{\ams}[1]
           {\begin{center}
            \begin{minipage}{315.83pt}
            \small
            \noindent 2010 Mathematics Subject Classification:~{\uppercase{#1}}
            \end{minipage}
            \end{center}
            \par\normalsize
           }

\begin{document}
\title{Asymptotic results for the number of Wagner's solutions to a generalised birthday problem
}
\author{Alexey Lindo and Serik Sagitov\\{\it Chalmers University of Technology and University of Gothenburg}
}
\date{}
\maketitle
\begin{abstract}
We study two functionals of a random matrix $\boldsymbol A$ with independent elements uniformly distributed over the cyclic group of integers $\{0,1,\ldots, M-1\}$ modulo $M$. One of them, $V_0(\boldsymbol A)$ with mean $\mu$,  gives the total number of solutions for a generalised birthday problem, and the other, $W(\boldsymbol A)$ with mean $\lambda$, gives the number of solutions detected by Wagner's tree based algorithm. 

We establish two limit theorems. Theorem \ref{pnm} describes an asymptotical behaviour of the ratio $\lambda/\mu$ as $M\to\infty$. Theorem \ref{TH} 
suggests Chen-Stein bounds for the total variation distance between Poisson distribution and distributions of $V_{0}$ and $W$.
\end{abstract}

\ams{60B20, 60C05, 60F05}

\section{Introduction}
Let $(N,M,L)$ be three natural numbers larger than or equal to 2. Assume that we have a random matrix 
\begin{equation}\label{mA}
\boldsymbol A=(a_{ij}), \ {1\le i\le L, \ 1\le j\le N}
\end{equation}
 with  independent elements $a_{ij}$ which are uniformly distributed on $\{0,1,\ldots, M-1\}$.
Let $\boldsymbol J = \{1,\ldots,L\}^{N}$ be the set of  matrix positions, so that  $|\boldsymbol J| = L^{N}$.
For each $b\in\{0,1,\ldots, M-1\}$, define $V_b\equiv V_b(\boldsymbol A)$ as the number of vectors $\boldsymbol i=(i_1,\ldots, i_{N})\in \boldsymbol J$ with
$$a_{i_1,1}+\ldots+a_{i_N,N}\stackrel{M}{=}b,$$
where the sign $\stackrel{M}{=}$ means equality modulo $M$.
Clearly, $\sum_{b=0}^{M-1} V_b=L^N$, so that by the assumption of uniform distribution,  
\[\mu:=\rE(V_0)=L^NM^{-1}.\]

The problem of finding all $V_0$ zero-sum vectors 
\begin{equation}\label{ai}
 \boldsymbol a_{\boldsymbol  i}=(a_{i_1,1},\ldots,a_{i_N,N}),\quad \boldsymbol i=(i_1,\ldots, i_{N})\in \boldsymbol J
\end{equation}
for a given matrix $\boldsymbol A$,  can be viewed as a generalised birthday problem. It arises naturally in a variety of situations including cryptography, see~\cite{Wagner2002} and reference therein; ring linear codes~\cite{Gr}; abstract algebra, where in the theory of modules it is related to the notion of annihilator, see e.g.~\cite{L2002}.
This problem can be solved only by exhaustive search and is $NP$-hard \cite{SSh}. Wagner \cite{Wagner2002} proposed a subexponential algorithm giving hope to quickly  detect at least some of the solutions to this kind of problems.

Assume that $N = 2^{n}$, $n \ge 1$ and $M = 2^{m} + 1$, $m \ge n$. It  will be convenient to use the symmetric form
  $$D_m:=\{ -2^{m - 1}, \ldots, -1,0,1,\ldots, 2^{m - 1} \}$$ 
of $\{0,1,\ldots, M-1\}$ as the set of possible values for $a_{ij}$.  Wagner's algorithm has a binary tree structure, see Figure \ref{waga}, starting from $N$ leaves at level $n$ and moving toward the top of the tree at level 0.  For a given a vector  $\boldsymbol x=(x_1,\ldots, x_{2^n})$ with $x_j\in D_m$ the algorithm searches for the value 
  \begin{equation}\label{hn}
  H_n(\boldsymbol x):=x_{1}^{(n)} \in D_{m-n} \cup\{\Delta\},
\end{equation}
obtained recursively in a way explained next (the special state $\Delta$ indicates that the algorithm is terminated and a solution is not found).  
 \begin{figure}[ht]
   \centering
     \includegraphics[width=8cm]{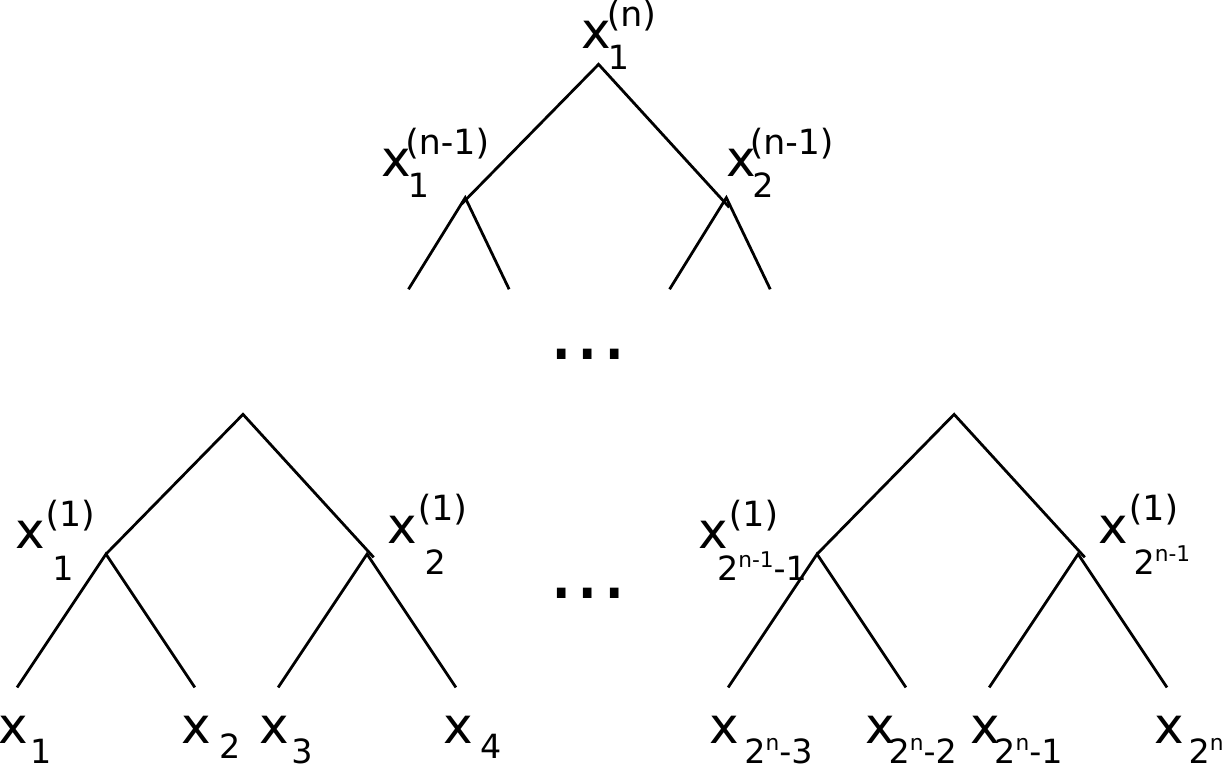}
   \caption{Wagner's algorithm}
   \label{waga}
 \end{figure}
Put $x_j^{(0)}\equiv x_j$. For $h=1,\ldots, n$ and $j=1,\ldots, 2^{n-h}$, let $x_j^{(h)}=b$ if there exists such a $b\in D_{m-h}$ that 
\[x_{2j-1}^{(h-1)}+x_{2j}^{(h-1)}\stackrel{M}{=}b,\]
and put $x_j^{(h)}=\Delta$ otherwise. In particular, if $x_{k}^{(h-1)}=\Delta$ for at least one of the two indices $k\in\{2j-1,2j\}$, then $x_j^{(h)}=\Delta$.

A vector  $\boldsymbol x$ will be called a Wagner's solution to the generalised birthday problem, if $H_n(\boldsymbol x) =0$. The total number $W\equiv W(\boldsymbol A)$ of Wagner's solutions among the vectors \eqref{ai} has mean
$$\lambda := \rE( W) = L^{N} p_{n,m},$$
where 
$$p_{n,m}:=\rP(H_n(\boldsymbol a_{\boldsymbol  i})=0),\quad \boldsymbol i\in \boldsymbol J.$$
The proportion of Wagner's solutions can be characterised by the ratio of the means 
\begin{equation}\label{rm}
R_{n,m}:=\lambda/\mu=(2^m+1)p_{n,m}. 
\end{equation}
Clearly, $R_{n,m}$ is the conditional probability of a given zero-sum random vector to be Wagner's solution.

There is a growing number of papers studying the properties of various tree based algorithms with some of them, in particular \cite{MiSi}, suggesting further developments of Wagner's approach. The main results of this paper are stated in the next section. Theorem \ref{pnm} gives an integral recursion for calculating the limit for the key ratio \eqref{rm}. Theorem \ref{TH}
suggests Chen-Stein bounds for the total variation distance between Poisson distribution and distributions of $V_{0}$ and $W$.
(Among related results concerning speed of convergence for functional of random matrices over finite algebraical structures we can only name a recent paper \cite{FG2015}.) 

\section{Main results}

Define a sequence of polynomials $\{\phi_n(x)\}_{n\ge1}$ by 
 \begin{equation}\label{recp}
\phi_n(x):=\int_0^x\phi_{n-1}(u)\phi_{n-1}(x-u)du+2\int_x^{2^{-n}}\phi_{n-1}(u)\phi_{n-1}(u-x)du,
\end{equation}
 with $\phi_1(x)\equiv1$.


\begin{thm} \label{pnm}
For any fixed natural number $n$, 
 \[R_{n,m}\to\phi_n(0),\quad m\to\infty,\]
 where the limit is obtained from the integral recursion \eqref{recp}.
 \end{thm}
To illustrate Theorem \ref{pnm}, take $N=16$, $L=1000$, and $M=10^{45}$. Then the expected number of zero-sum vectors is  $\mu=1000$. In practice, finding all zero-sum vectors out of $L^N=10^{48}$ candidates is a time consuming task. 
In this example we have $n=4$ and $m$ is approximately 150. Judging from Figure \ref{latomu} illustrating the typical values for the proportion factor $R_{n,m}$ using numerical computations based on the recursions for \eqref{Rnm} presented in the next section, out of a thousand solutions the Wagner algorithm will catch no more than one.

\begin{figure}[ht]
  \centering
    \includegraphics[width=10cm, height=8cm]{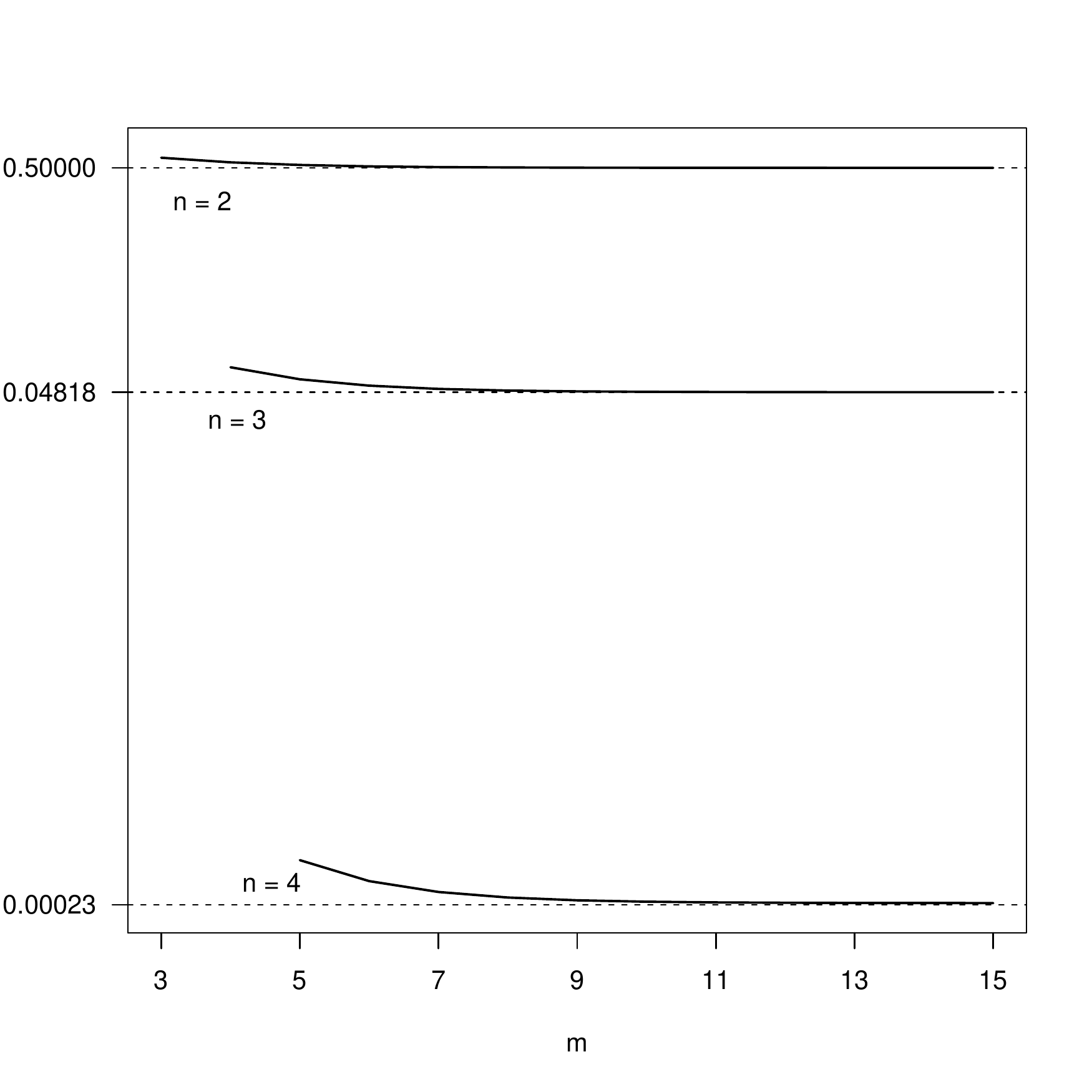}
  \caption{The ratio of the means \eqref{rm} for $n=2,3,4$ are plotted as functions of $m$. The limits predicted by Theorem \ref{pnm} are indicated by horizontal dotted lines.}
  \label{latomu}
\end{figure}

\begin{thm}\label{TH}
For a random matrix \eqref{mA} consider the number $V_0$ of vectors \eqref{ai} such that $a_{i_1,1}+\ldots+a_{i_N,N}\stackrel{M}{=}0$.
Then
  \begin{equation*} 
    \sum_{k = 0}^{\infty} \Big| \rP(V_0 = k) - \frac{\mu^{k} \re^{-\mu}}{k!} \Big| \le 4(1-\re^{-\mu} )M^{-1},
  \end{equation*}
where $\mu=L^NM^{-1}$. Furthermore, if  $N = 2^{n}$ and $M=2^m+1$, $m > n$, then with $\lambda=L^Np_{n,m}$
  \begin{equation*}
    \sum_{k = 0}^{\infty} \Big| \rP(W = k) - \frac{\lambda^{k} \re^{-\lambda}}{k!} \Big| \le 8 (1-\re^{-\lambda} )\mu  NL^{-1}.
  \end{equation*}
\end{thm}

According to Theorem \ref{TH}, Poisson approximation for $V_0$ works well when $L^N\ll M$. For $W$, a sufficient condition for the Chen-Stein bound to be small is $NL^{N-1}\ll M$. 

\section{Key recursion}
Consider a backward recursion
\begin{equation}\label{recu}
 v_{i}(j) = \sum_{k = 0}^{j} v_{i+1}(k) v_{i+1}(j-k) + 2\sum_{k = j+1}^{2^{i}} v_{i+1}(k) v_{i+1}(k-j)
\end{equation}
involving a system of vectors $(v_i(0), \ldots, v_i(2^{i-1}))$ for $i\ge1$. In particular, we have 
\begin{align*}
 v_{i}(0) &= v_{i+1}^2(0) + 2\sum_{k = 1}^{2^{i}} v_{i+1}^2(k).
\end{align*}
For $ 1\le i\le m-1$, denote by $v_i^{(m)}(j)$ the unique solution of \eqref{recu} determined by the following frontier condition
  \begin{equation*}
    v_{m-1}(0) = \cdots = v_{m-1}(2^{m - 2}) = (1+2^m)^{-1}.
  \end{equation*}
By the forthcoming Corollary \ref{cor}, we can write $p_{n,m}=v_{m-n}^{(m)}(0)$ so that 
 \begin{equation}\label{Rnm}
R_{n,m}=(1+2^m)v_{m-n}^{(m)}(0),\quad n=1, \ldots,m-1.
\end{equation}

\begin{lem} \label{thm:pi}
Let  $1\le n\le m-1$ and $H_n(\boldsymbol x)$ be defined by \eqref{hn}. Assuming that $\boldsymbol x$ is a random vector with independent component uniformly distributed over $D_m$, put
\[p_{i,m}(j) := \rP(H_i(\boldsymbol x)\stackrel{M}{=}j).\]
Then
  \begin{equation*} 
    p_{1,m}(-2^{m - 2}) = \cdots = p_{1,m}(2^{m - 2}) = (2^m+1)^{-1},
  \end{equation*}
 and for $2\le i \le m-1$ and  $0 \le j \le 2^{m - i - 1}$, we have $ p_{i,m}(-j) = p_{i,m}(j) $ with $p_{i,m}(j) $ satisfying the recursion
   \[ p_{i,m}(j) = \sum_{k = 0}^{j} p_{i-1,m}(k) p_{i-1,m}(j-k)+2\sum_{k = j+1}^{2^{m - i }} p_{i-1,m}(k) p_{i-1,m}(k-j).
 \]
\end{lem}
\begin{proof}
  There are exactly $M=2^m+1$ different ordered pairs of numbers from the set $D_m$  that add modulo $M$ up to a given $j\in D_{m-1}$.  These pairs have the form:   for $j=0$, 
    \[(-2^{m-1}+k,2^{m-1}-k), k=0,\ldots, 2^{m},\]
    for $j=1,\ldots,2^{m-2}$, 
  \begin{align*}
(-2^{m-1}+k,-2^{m-1}+j-k-1), \quad &k=0,\ldots, j-1,\\
(-2^{m-1}+k,2^{m-1}+j-k), \quad &k=j,\ldots, 2^{m},
\end{align*}
   and   for $j=-2^{m-2},\ldots,-1$,
   \begin{align*}
   (2^{m-1}-k,2^{m-1}+j+k+1), \quad &k= 0,\ldots,|j|-1,\\
    (2^{m-1}-k, -2^{m-1}+j+k), \quad &k=|j|,\ldots, 2^{m}.
\end{align*}
  Since these pairs appear with equal probability $M^{-2}$, the first claim follows.

On the other hand, for a given $j\in D_{m-i}$ with $i\ge2$, there are only $M-|j|$ different ordered pairs of numbers from the set $D_{m-i+1}$  that add modulo $M$ up to $j$.  These pairs have the form:   
\begin{align*}
& (-2^{m-i}+k,2^{m-i}+j-k),\quad k=j,\ldots, 2^{m-i+1},\quad j=0,\ldots,2^{m-i-1},\\
&(2^{m-i}-k,-2^{m-i}+j+k),\quad k=|j|,\ldots, 2^{m-i+1},\quad j=-2^{m-i-1},\ldots,-1.
\end{align*}
  This yields for $ j=1,\ldots,2^{m-i-1}$,
\begin{align*}
p_{i,m}(j) &= \sum_{k = j}^{2^{m - i +1}} p_{i - 1,m}(-2^{m - i} + k) p_{i - 1,m}(2^{m - i} - k + j),\\
p_{i,m}(-j) &= \sum_{k = j}^{2^{m - i +1}} p_{i - 1,m}(2^{m - i} - k) p_{i - 1,m}(-2^{m - i} +k - j).
\end{align*}
 The stated symmetry property $ p_{i,m}(-j) = p_{i,m}(j) $ now follows recursively from the assumption of uniform distribution. To finish the proof of the lemma, it remains to observe that after replacing $k-2^{m - i}$ by $l$ in the last relation for $p_{i,m}(j) $ we get
 \begin{align*} 
   p_{i,m}(j) &= \sum_{l = j-2^{m - i}}^{2^{m - i }} p_{i-1,m}(l) p_{i-1,m}(j-l),
  \end{align*}
  which in turn equals to 
 \begin{align*} 
 \sum_{l = 0}^{j} p_{i-1,m}(l) p_{i-1,m}(j-l)&+\sum_{l = j+1}^{2^{m - i }} p_{i-1,m}(l) p_{i-1,m}(l-j)+ \sum_{l =j-2^{m - i}}^{-1} p_{i-1,m}(-l) p_{i-1,m}(j-l)\\
   &=\sum_{k= 0}^{j} p_{i-1,m}(k) p_{i-1,m}(j-k)+2\sum_{k = j+1}^{2^{m - i}} p_{i-1,m}(k) p_{i-1,m}(k-j).
  \end{align*}
 \end{proof}
\begin{corollary}\label{cor}
 Comparison of the key recursion in Lemma \ref{thm:pi} with the recursion \eqref{recu} yields
 \[p_{m-i,m}(j)=v_i^{(m)}(j).\]
\end{corollary}

 \section{Proof of Theorem \ref{pnm}}

 Recall \eqref{Rnm} and put 
 $$R_{n,m}(j)=2^mv_{m-n}^{(m)}(j),\qquad \phi_{n,m}(x):=\phi_n(x2^{-m}).$$
  We prove Theorem \ref{pnm} by verifying a more general convergence result
  
\begin{equation}\label{unif}
  \alpha_{n,m}:=\max_{0\le j\le 2^{m-n-1}}|R_{n,m}(j)-\phi_{n,m}(j)|\to0,\quad m\to\infty.
\end{equation}
  To this end we use induction over $n$. The base case $n=1$ is trivial. To prove the inductive step observe first that by \eqref{recu}
 \begin{equation}\label{ru}
R_{n,m}(j) = 2^{-m}\sum_{k = 0}^{j} R_{n-1,m}(k) R_{n-1,m}(j-k) + 2^{1-m}\sum_{k = j+1}^{2^{m-n}} R_{n-1,m}(k) R_{n-1,m}(k-j).
\end{equation}
It is easy to see recursively that the constant
 \[ C_n:=\sup_{m>n}\max_{0 \le j \le 2^{m-n-1}}R_{n,m}(j)\]
is finite. 

On the other hand, by \eqref{recp},
 \[\phi_{n,m}(j)=2^{-m}\int_0^{j}\phi_{n-1,m}(u)\phi_{n-1,m}(j-u)du+2^{1-m}\int_j^{2^{m-n}}\phi_{n-1,m}(u)\phi_{n-1,m}(u-j)du,\]
so that
   \begin{align}
\phi_{n,m}(j)&= 2^{-m}\sum_{k = 0}^{j}  \phi_{n-1,m}(k)  \phi_{n-1,m}(j-k) \nonumber\\
&\quad+ 2^{1-m}\sum_{k = j+1}^{2^{m-n}}  \phi_{n-1,m}(k)  \phi_{n-1,m}(k-j)+\epsilon_{n,m}(j), \label{rup}
\end{align}
with accordingly defined remainder term $\epsilon_{n,m}(j)$.
Uniform continuity of $\phi_n(x)$ yields uniform convergence $\epsilon_{n,m}(j)\to0$ as $m\to\infty$, and  \eqref{unif} follows from \eqref{ru} and \eqref{rup}, since 
 \[ \alpha_{n,m}\le 2\big[C_{n-1}+\max_{0\le x\le 2^{-n}}\phi_{n}(x)\big]\alpha_{n-1,m}+\max_{0\le j\le 2^{m-n}}|\epsilon_{n,m}(j)|.\]

%
%

\section{Proof of Theorem \ref{TH}}


The following result is a straightforward corollary of Theorem 1 from \cite{Arratia1989} and is a key tool for our proof here.
\begin{lem} \label{CS}
Let $Z=\sum_{\boldsymbol{i}\in  \boldsymbol J} \chi_{\boldsymbol{i}}$ be a sum of possibly dependent indicator random variables with $\rE(Z)=\zeta$. Suppose there is a family of subsets $\boldsymbol J_{\boldsymbol{i}}\subset\boldsymbol J$ such that for any $\boldsymbol{i} \in \boldsymbol J$ and $\boldsymbol{k} \notin \boldsymbol J_{\boldsymbol{i}}$, indicators $\chi_{\boldsymbol{i}}$ and $\chi_{\boldsymbol{k}}$ are independent. Then
\[
\frac{\zeta}{4(1 - \re^{-\zeta})}\sum_{k = 0}^{\infty} \Big| \rP(Z = k) - \frac{\zeta^{k} \re^{-\zeta}}{k!} \Big| \le 
\sum_{\boldsymbol{i} \in \boldsymbol J} \sum_{\boldsymbol{k} \in \boldsymbol J_{\boldsymbol{i}}} \rE \big( \chi_{\boldsymbol{i}} \big) \rE \big( \chi_{\boldsymbol{k}} \big)
+\sum_{\boldsymbol{i} \in \boldsymbol J} \sum_{\boldsymbol{k} \in \boldsymbol J_{\boldsymbol{i}}\setminus \{\boldsymbol{i} \}}  \rE \big( \chi_{\boldsymbol{i}} \chi_{\boldsymbol{k}} \big).
\]
\end{lem}

We start the proof of Theorem \ref{TH}
by observing that $V_0=\sum_{\boldsymbol{i}\in  \boldsymbol J} \chi_{\boldsymbol{i}}$, where the indicator random variables
$$\chi_{\boldsymbol{i}}=1_{\{a_{i_1,1}+\ldots+a_{i_N,N}\stackrel{M}{=}0\}},\quad \boldsymbol i=(i_1,\ldots, i_{N})$$
are identically distributed  with $E(\chi_{\boldsymbol{i}})=M^{-1}$,  and mutually independent. Independence is due to the 
defining property of the matrix $\boldsymbol A$. Indeed, if $\boldsymbol{k} \ne \boldsymbol{i} $ and (without loss of generality) $1,\ldots,j$ are the coordinates where these two vectors differ, then 
\begin{align*}
 \rP(a_{k_1,1}&+\ldots+a_{k_N,N} \stackrel{M}{=} a_{i_1,1}+\ldots+a_{i_N,N} \stackrel{M}{=}0) \\
 &=\rP(a_{k_1,1}+\ldots+a_{k_j,j} \stackrel{M}{=} a_{i_1,1}+\ldots+a_{i_j,j}\stackrel{M}{=}- a_{i_{j+1},j+1}-\ldots-a_{i_N,N})\\
 &=\sum_{b\in D_m}\rP(a_{k_1,1}+\ldots+a_{k_j,j} \stackrel{M}{=} b; a_{i_1,1}+\ldots+a_{i_j,j}\stackrel{M}{=}b; a_{i_{j+1},j+1}+\ldots+a_{i_N,N}\stackrel{M}{=}- b)\\
 &=M^{-1}\sum_{b\in D_m}\rP(a_{i_1,1}+\ldots+a_{i_j,j}\stackrel{M}{=}b; a_{i_{j+1},j+1}+\ldots+a_{i_N,N}\stackrel{M}{=}- b)= M^{-2}. 
\end{align*}
Therefore, we can apply  Lemma \ref{CS} with $\boldsymbol J_{\boldsymbol{i}}=\{\boldsymbol{i} \}$, and the Chen-Stein bound for $V_0$ follows from $\rE(V_0)=\mu$ and
\[ \sum_{\boldsymbol{i} \in \boldsymbol J} \sum_{\boldsymbol{k} \in B_{\boldsymbol{i}}} \rE \big( \chi_{\boldsymbol{i}} \big) \rE \big( \chi_{\boldsymbol{k}} \big)
 = L^{N} M^{-2}=\mu M^{-1}. \]

To obtain the Chen-Stein  bound for $W$, we define $\boldsymbol J_{\boldsymbol{i}}$ as the set of $\boldsymbol{k} \in L$ such that  vectors $\boldsymbol{i}$ and $\boldsymbol{k} $ share at least one component. Observe that 
\[|\boldsymbol J_{\boldsymbol{i}}| = L^{N} - (L - 1)^{N}. \]
By definition of $W$, 
$$W=\sum_{\boldsymbol{i}\in  \boldsymbol J} \chi_{\boldsymbol{i}},\quad \chi_{\boldsymbol{i}}=1_{\{H_n(\boldsymbol a_{\boldsymbol  i})=0\}},$$ 
so that $\rE \big( \chi_{\boldsymbol{i}}\big) = p_{n,m}$ and therefore,
\begin{equation*} \label{eq:b1}
\sum_{\boldsymbol{i} \in \boldsymbol J} \sum_{\boldsymbol{k} \in \boldsymbol J_{\boldsymbol{i}}} \rE \big( \chi_{\boldsymbol{i}} \big) \rE \big( \chi_{\boldsymbol{k}} \big)
 = L^{N} \big( L^{N} - (L - 1)^{N} \big) p_{n,m}^{2} \le NL^{- 1} \lambda^2.
\end{equation*}
Since a Wagner's solution is necessarily is a zero-sum vector, we have for $\boldsymbol  i\ne\boldsymbol  k$,
\begin{align*}
 \rE \big( \chi_{\boldsymbol{i}} \chi_{\boldsymbol{k}} \big)= \rP(H_n&(\boldsymbol a_{\boldsymbol  i})=0; H_n(\boldsymbol a_{\boldsymbol  k})=0)\le \rP(a_{k_1,1}+\ldots+a_{k_N,N} \stackrel{M}{=} 0; H_n(\boldsymbol a_{\boldsymbol  i})=0). 
\end{align*}
Let $l_1,\ldots,l_j$ are the coordinates where the vectors $\boldsymbol  i, \boldsymbol  k$ differ. Then it follows that
\begin{align*}
 \rE \big( \chi_{\boldsymbol{i}} \chi_{\boldsymbol{k}} \big) &\le\sum_{b\in D_m}\rP(a_{k_{l_1},l_1}+\ldots+a_{k_{l_j},l_j} \stackrel{M}{=} b; a_{i_{l_1},l_1}+\ldots+a_{i_{l_j},l_j}\stackrel{M}{=}b; H_n(\boldsymbol a_{\boldsymbol  i})=0)\\
  &=M^{-1}\sum_{b\in D_m}\rP(a_{i_{l_1},{l_1}}+\ldots+a_{i_{l_j},lj}\stackrel{M}{=}b; H_n(\boldsymbol a_{\boldsymbol  i})=0)= M^{-1}p_{n,m},
\end{align*}
and we get
\begin{equation*} \label{eq:b2}
\sum_{\boldsymbol{i} \in \boldsymbol J} \sum_{\boldsymbol{k} \in \boldsymbol J_{\boldsymbol{i}}\setminus \{\boldsymbol{i} \}}  \rE \big( \chi_{\boldsymbol{i}} \chi_{\boldsymbol{k}} \big)
\le  L^{N} \big( L^{N} - (L - 1)^{N}  \big) p_{n,m} M^{-1}\le NL^{- 1} \lambda\mu.
\end{equation*}
The proof is finished by applying once again  Lemma \ref{CS}.

\vspace{1cm}
\noindent{\bf Acknowledgements}. The first author is grateful to Vladimir Vatutin and Andrey Zubkov for formulating an initial problem setting that eventually lead to this research project.

\end{document}